\documentclass[a4paper,reqno]{amsart}

\usepackage{amsmath}
\usepackage[english]{babel}
\usepackage{amssymb}
\usepackage{amsthm}
\usepackage{thmtools}
\usepackage{lmodern}
\usepackage{verbatim}
\usepackage[T1]{fontenc}
\usepackage{amsfonts}
\usepackage{amsmath}
\usepackage{amssymb}
\usepackage{eurosym}
\usepackage{mathrsfs}
\usepackage{color}
\usepackage{palatino}
\usepackage{enumerate}
\usepackage{url}

\def\Xint#1{\mathchoice{\XXint\displaystyle\textstyle{#1}}%
	{\XXint\textstyle\scriptstyle{#1}}%
	{\XXint\scriptstyle\scriptscriptstyle{#1}}%
	{\XXint\scriptscriptstyle\scriptscriptstyle{#1}}%
	\!\int}
\def\XXint#1#2#3{{\setbox0=\hbox{$#1{#2#3}{\int}$ }
		\vcenter{\hbox{$#2#3$ }}\kern-.6\wd0}}

\def\dashint{\Xint-}

\newcommand{\avgint}[2]{\dashint_{#2} {#1}}

\newcommand{\R}{\mathbb{R}}
\newcommand{\C}{\mathbb{C}}

\newcommand{\N}{\mathbb{N}}
\newcommand{\Z}{\mathbb{Z}}

\newcommand{\bla}{\big \langle}
\newcommand{\bra}{\big \rangle}

\newcommand{\ud}[0]{\,\mathrm{d}}

\newcommand{\esssup}[0]{\operatornamewithlimits{ess\,sup}}

\newcommand{\abs}[1]{|#1|}
\newcommand{\Babs}[1]{\Big|#1\Big|}

\newcommand{\bnorm}[2]{\big|#1\big|_{#2}}
\newcommand{\Bnorm}[2]{\Big|#1\Big|_{#2}}
\newcommand{\Norm}[2]{\|#1\|_{#2}}
\newcommand{\bNorm}[2]{\big\|#1\big\|_{#2}}

\newcommand{\ave}[1]{\langle #1\rangle}

\newcommand{\BMO}[0]{\operatorname{BMO}}

\newcommand{\supp}[0]{\operatorname{spt}}
\newcommand{\loc}[0]{\operatorname{loc}}

\allowdisplaybreaks
\theoremstyle{definition}
\newtheorem{definition}[subsection]{Definition}

\newtheorem{proposition}[subsection]{Proposition}
\newtheorem{corollary}[subsection]{Corollary}
\newtheorem{conjecture}[subsection]{Conjecture}

\theoremstyle{plain}

\theoremstyle{remark}
\newtheorem{remark}[subsection]{Remark}

\theoremstyle{theorem}
\newtheorem{theorem}[subsection]{Theorem}
\newtheorem{lemma}[subsection]{Lemma}
\DeclareMathOperator{\sgn}{sgn}
\numberwithin{equation}{section}

\title[Iterated commutators under a joint condition on the tuple of multiplying functions]{Iterated commutators under a joint condition on the tuple of multiplying functions}

\author{Tuomas Hyt\"onen}
\address[T. Hyt\"onen]{Department of Mathematics and Statistics, P.O. Box 68 (Pietari Kalmin katu 5),
FI-00014 University of Helsinki, Finland}
\email{tuomas.hytonen@helsinki.fi}
\author{Kangwei Li}
\address[K. Li]{Center for Applied Mathematics, Tianjin University, Weijin Road 92, 300072 Tianjin, China}
\email{kli@tju.edu.cn}
\author{Tuomas Oikari}
\address[T. Oikari]{Department of Mathematics and Statistics, P.O. Box 68 (Pietari Kalmin katu 5),
FI-00014 University of Helsinki, Finland}
\email{tuomas.v.oikari@helsinki.fi}

\makeatletter
\@namedef{subjclassname@2010}{%
	\textup{2010} Mathematics Subject Classification}
\makeatother

\thanks{T.H. and T.O. are members of the Finnish Centre of Excellence in Analysis and Dynamics Research supported by the Academy of Finland (project No. 307333). T.H. is also supported by the Academy of Finland project No. 314829 and T.O. by the three-year research grant 75160010 of the University of Helsinki and the Academy of Finland
project No. 306901.
}

\begin{document}
	
\maketitle

\begin{abstract}
	
We present a pair of joint conditions on the two functions $b_1,b_2$ strictly weaker than $b_1,b_2\in\BMO$ that almost characterize the $L^2$ boundedness of the iterated commutator $[b_2,[b_1,T]]$ of these functions and a Calder\'on-Zygmund operator $T.$ Namely, we sandwich this boundedness between two bisublinear mean oscillation conditions of which one is a slightly bumped up version of the other.	

\end{abstract}

\section{Introduction}

	The study of commutators of Calder\'on-Zygmund operators with pointwise multiplication has been a long standing interest in the field of harmonic analysis; for example, in the fundamental paper of Coifman, Rochberg, Weiss \cite{CRW} a characterization of the space $\BMO(\R^d)$ is given with respect to the commutators taken with the Riesz transforms:
	$$
	[b,R_j]:L^2(\R^d)\to L^2(\R^d)\quad\mbox{ boundedly for all} \quad j=1,\dots,d\quad
	$$
	if and only if $b\in\BMO(\R^d).$ Here $[b,R_j] = bR_j - R_j(b\,\cdot)$.
	Already in \cite{CRW} it was shown that $b\in\BMO$ is a sufficient condition for the boundedness of the iterated commutator $[b,[b,\dots,[b,T]]]$ of pointwise multiplications and a Calder\'on-Zygmund operator and the same argument extends to the case of commutators $[b_k,[b_{k-1},\dots,[b_1,T]]]$ with different functions, all in $\BMO$ separately.
	
	Our object is to make the first systematic study of the iterated commutator $[b_2,[b_1,R_j]]$
	in the case of two different functions $b_1,b_2.$ In particular, we want to identify a joint condition on the pair $(b_1, b_2)$ that is weaker than the individual conditions
	$b_1, b_2 \in \BMO$, that is as close to optimal as possible, and which still guarantees the boundedness of the commutator.
	 This is, in some sense, similar in spirit to the case of bilinear weighted theory, where $w_1, w_2 \in A_4$ is not the optimal condition
	for the boundedness of bilinear singular integrals from $L^4(w_1) \times L^4(w_2)$ to $L^2(w_1^{1/2}w_2^{1/2})$, but rather there
	is a genuinely bilinear joint condition $(w_1, w_2) \in A_{(4,4)}$ introduced by Lerner, Ombrosi, P\'erez, Torres and Trujillo-Gonz\'alez \cite{LOPTT}. In the weighted
	case the identification of this genuinely bilinear condition has been highly impactful.
		
	We study two-sided estimates for the $L^2 \to L^2$ norm of the commutator $[b_2,[b_1,T]]$. While the upper bounds will be valid for all bounded singular integrals,
	the lower bounds require some suitable non-degeneracy, and here we work with the Riesz
	 transforms
	\[
	R_jf(x) = \lim_{\varepsilon\to 0}\int_{\abs{x-y}>\varepsilon}\frac{x_j-y_j}{|x-y|^{d+1}}f(y)\ud y, \quad f\in L^2(\R^d),\quad \ j=1,\cdots, d.
	\]	
	We show that
	\begin{equation}\label{intro1}
	C_d(S_{2}(b_1,b_2)+T_{2}(b_1,b_2))  \leq \bNorm{[b_2,[b_1,T]]}{L^2(\R^d)\to L^2(\R^d)}\leq  C_{T,\varepsilon}(S_{2+\varepsilon}(b_1,b_2)+T_{2+\varepsilon}(b_1,b_2)),
	\end{equation}
	where the constant $C_{T,\varepsilon}$ tends to infinity as $\varepsilon$ tends to zero and the joint conditions $S_p,T_p$, with $0<p<\infty,$ imposed on the complex valued functions $b_1,b_2$ are defined by
	\begin{equation*}\label{intro2}
	S_p(b_1,b_2)  = \sup_{Q}\left(\frac{1}{\abs{Q}}\int_Q\abs{b_1-\langle b_1\rangle_Q}^p\right)^{1/p}\left(\frac{1}{\abs{Q}}\int_Q\abs{b_2-\langle b_2\rangle_Q}^p\right)^{1/p}
	\end{equation*}
	and
	\begin{equation*}\label{intro3}
	T_p(b_1,b_2)  = \sup_{Q}\left(\frac{1}{\abs{Q}}\int_Q\abs{b_1-\langle b_1\rangle_Q}^p\abs{b_2-\langle b_2\rangle_Q}^p\right)^{1/p},\qquad b_1,b_2,b_1b_2\in L^p_{loc}(\R^d).
	\end{equation*}
	Here the supremums are taken over all cubes. Whenever it is well understood which functions $b_1,b_2$ are in question, we refer to these conditions shortly as $T_p$ and $S_p.$
	
	We show by example  that the lower bound in \eqref{intro1} does not improve to $S_{2+\varepsilon}(b_1,b_2)+T_{2+\varepsilon}(b_1,b_2)$ for any $\varepsilon>0$ -- that is,
	the obtained upper bound is not necessary.
	This leads us to consider joint conditions involving Young functions that can be made strictly weaker than $S_{2+\varepsilon}+T_{2+\varepsilon}$ for all $\varepsilon >0.$
	Hence, we prove the commutator upper bound with these updated conditions with a version of the sparse domination principle introduced in Lerner \cite{L}.

\subsection{Basic notation}
	We denote $A \lesssim B$, if $A \leq C B$ for some constant $C>0$ depending only on the dimension of the underlying space, on integration exponents and on other concurrently unimportant absolute constants appearing in the assumptions.
	Then naturally $A \sim B$, if $A \lesssim B$ and $B \lesssim  A.$ Subscripts on constants ($C_{a,b,c,...}$) signify their dependence on those subscripts.
	
	We also denote the space $L^p(\R^d)$ with $L^p.$
	
	Integral average is by dash or brackets: $\frac{1}{\abs{Q}}\int_Qf = \avgint{f}{Q} = \langle f\rangle_Q.$
	
	When we say that an operator $A$ is bounded on $L^p$ we mean that $A:L^p\to L^p$ boundedly.
	
\subsection*{Acknowledgements}
We thank Henri Martikainen for posing research questions for this paper and for valuable discussions and comments. We also thank the anonymous referee for constructive comments that improved the presentation.

\section{Necessary conditions}
We move on to derive the lower bound $S_2+T_2$ for the iterated commutator taken with the Riesz transforms.
Later we see that the condition $S_2+T_2$ is not strong enough to imply the $L^2$ boundedness of the commutator, however.

Before proceeding any further, let us precisely define the commutator $[b_2,[b_1,T]].$
\begin{definition}\label{comdef} Let $b_i,$ $i=1,2,$ be such that $b_1,b_2,b_1b_2 \in L^2_{loc}(\R^d)$ and denote $b=(b_1,b_2).$  With $T$ being an operator on $L^2(\R^d)$, the commutator $C_bT$ on $L^{\infty}_c(\R^d)$ is defined as
	\begin{equation*}\label{basicdef}
	C_bT= [b_2,[b_1,T]],
	\end{equation*}
	where $[A,B] = AB-BA$ for any two operations $A,B,$ and $b_if(\cdot) = b_i(\cdot)f(\cdot).$
\end{definition}
We deal with the second order commutator $[b_2,[b_1,T]]$ but our results concerning sufficient conditions could just as well be formulated in the higher order cases.

\begin{lemma}\label{rieszlemma} Let $R_j$ be the jth Riesz transform on $\R^d,$ $j=1,\dots,d,$ $f_1,f_2\in L^{\infty}_c$ and $b_1,b_2,b_1b_2 \in L^1_{loc}.$ Under these assumptions, for all cubes $Q,$ we have that
	\begin{equation*}
	\Bnorm{\avgint{\avgint{\prod_{i=1}^{2}(b_i(x)-b_i(y))f_2(y)f_1(x) \ud y \ud x}{Q}}{Q}}{} \leq C_d\sum_{i=1}^{k}\bNorm{C_bR_i}{L^p\to L^p} \left( \avgint{\abs{f_1}^p}{Q} \right) ^{1/p}\left( \avgint{\abs{f_2}^{p'}}{Q} \right) ^{1/p'}.
	\end{equation*}
\end{lemma}
\begin{proof}
	Our proof separates into two cases, to odd and even dimensions.
	
	\textit{Case 1, $d$ is odd}:
	Let $d = 2k+1$ for some $k\in\N$.
	By composing back and forth with the translation $x\mapsto x-c_Q$, we may assume that the cube $Q$ is centred at the origin.
	We begin with introducing $1$ as
	\begin{equation}
	1=\sum_{i=1}^d\frac{(x_i-y_i)^2}{\abs{x-y}^{2}}
	=\sum_{i=1}^d\frac{(x_i-y_i)}{\abs{x-y}^{d+1}}(x_i-y_i)\abs{x-y}^{d-1},
	\end{equation}
	denote $b(x,y) = \prod_{i=1}^{2}(b_i(x)-b_i(y)),$
	and then proceed with:
	\begin{equation*}
	\begin{split}
	& \Bnorm{\int_Q\int_Qb(x,y)f_2(y)f_1(x)\ud y \ud x}{} = \Bnorm{\int_Q\lim_{\varepsilon\to 0}\int_{\abs{x-y}>\varepsilon}b(x,y)f_2(y)1_Q(y)f_1(x)\ud y \ud x}{} \\	
	&= \Bnorm{\int_Q\lim_{\varepsilon\to 0}\int_{\abs{x-y}>\varepsilon}b(x,y) \sum_{i=1}^d\frac{(x_i-y_i)^2}{\abs{x-y}^2} f_2(y)1_Q(y)f_1(x)\ud y \ud x}{} \\
	&= \Bnorm{\int_Q\lim_{\varepsilon\to 0}\int_{\abs{x-y}>\varepsilon}b(x,y) \sum_{i=1}^d\frac{Y_i(x-y)^2}{\abs{x-y}^{2d}} f_2(y)1_Q(y)f_1(x)\ud y \ud x}{},  \\
	\end{split}
	\end{equation*}
	where $Y_i(x) = x_i\abs{x}^{d-1}.$
	We momentarily force the expression into this form in order to contrast it with the similar argument emplying spherical harmonics given in \cite{CRW}.
	
	For a given $\alpha = (\alpha_1,\dots,\alpha_d)\in\N^d$ and $x\in\R^d,$ let $x^{\alpha} = \prod_{i=1}^dx_i^{\alpha_i}.$
	We continue with
	\begin{align*}
	& \Bnorm{\int_Q\lim_{\varepsilon\to 0}\int_{\abs{x-y}>\varepsilon}b(x,y) \sum_{i=1}^d\frac{Y_i(x-y)^2}{\abs{x-y}^{2d}} f_2(y)1_Q(y)f_1(x)\ud y \ud x}{}  \\
	&= \Bnorm{\int_Q\lim_{\varepsilon\to 0}\int_{\abs{x-y}>\varepsilon}b(x,y) \sum_{i=1}^d\frac{x_i-y_i}{\abs{x-y}^{d+1}}(x_i-y_i)\Big(\sum_{j=1}^d(x_j-y_j)^2\Big)^{k} f_2(y)1_Q(y)f_1(x)\ud y \ud x}{}  \\
	&= \Bnorm{\int_Q\lim_{\varepsilon\to 0}\int_{\abs{x-y}>\varepsilon}\sum_{i=1}^d\sum_{\alpha+ \beta= d}b(x,y)\frac{x_i-y_i}{\abs{x-y}^{d+1}} a_{\alpha,\beta}f_1(x)x^{\alpha}y^{\beta}  f_2(y)1_Q(y)\ud y \ud x}{}  \\
	& \overset{\ast}{=} \Bnorm{\sum_{i=1}^d\sum_{\alpha+ \beta= d}a_{\alpha,\beta}\int_Qf_1(x)x^{\alpha}\lim_{\varepsilon\to 0}\int_{\abs{x-y}>\varepsilon}b(x,y) \frac{x_i-y_i}{\abs{x-y}^{d+1}}y^{\beta} f_2(y)1_Q(y)\ud y \ud x}{}  \\
	&\leq \sum_{i=1}^d\sum_{\alpha+ \beta= d}\abs{a_{\alpha,\beta}}\bNorm{(\cdot)^{\alpha}f_1}{L^{p'}(Q)}\bNorm{C_bR_i((\cdot)^{\beta}f_21_Q)}{L^p(Q)} \\
	&\leq \sum_{i=1}^d\sum_{\alpha+ \beta= d}\abs{a_{\alpha,\beta}}\bNorm{(\cdot)^{\alpha}}{L^{\infty}(Q)}\bNorm{f_1}{L^{p'}(Q)}\bNorm{(\cdot)^{\beta}}{L^{\infty}(Q)}\bNorm{C_bR_i}{L^p\to L^p}\bNorm{f_2}{L^p(Q)} \\
	&\leq C_d\sum_{i=1}^d\sum_{\alpha + \beta = d}\abs{a_{\alpha\beta}} \abs{Q}\bNorm{C_bR_i}{L^p\to L^p}\bNorm{f_1}{L^{p'}(Q)}\bNorm{f_2}{L^p(Q)}\\
	&\leq C_d\abs{Q}\sum_i^d\bNorm{C_bR_i}{L^p\to L^p}\bNorm{f_1}{L^{p'}(Q)}\bNorm{f_2}{L^p(Q)},
	\end{align*}
	where at $\ast$ we used the fact that the limits exist separetely as $R_i((\cdot)^{\beta}1_Qf_2b(x,\cdot))(x),$
	and where in the second to last esimate we used the assumption that $Q$ is centered at the origin.
	Dividing by $\abs{Q}^2$ gives the claim.
	
	\textit{Case 2, $d\geq2$}:
	In the previous estimate we saw that the key issue with the lower bound for $C_b R_i$'s is the following:	
	We introduce $1$ as $\sum_{i=1}^d(x_i-y_i)^2\abs{x-y}^{-2}$, and would like to view this as $(x_i-y_i)\abs{x-y}^{-d-1}$ times functions that depend only on $x$ and only on $y$. As we saw:
	\begin{equation}\label{eq:formula}
	1=\sum_{i=1}^d\frac{(x_i-y_i)^2}{\abs{x-y}^{2}}
	=\sum_{i=1}^d\frac{(x_i-y_i)}{\abs{x-y}^{d+1}}(x_i-y_i)\abs{x-y}^{d-1},
	\end{equation}
	and the problem becomes about expanding $\abs{x-y}^{d-1}$ when $d$ is even, hence $d$ was odd.
	
	Consider the function $z_i\abs{z}^{d-1}$ of $z\in\R^d$. By induction, we check that $\partial^\alpha(z_i\abs{z}^{d-1})$ is a linear combination of terms of the form $z^\beta\abs{z}^{d-\abs{\alpha}-\abs{\beta}}$, where $\abs{\beta}\leq\abs{\alpha}+1$. In particular, when $\abs{\alpha}=d+1$, then $\partial^\alpha(z_i\abs{z}^{d-1})$ is a linear combination of terms of the form $z^\beta\abs{z}^{-1-\abs{\beta}}$. In particular, $\abs{\partial^\alpha(z_i\abs{z}^{d-1})}\lesssim\abs{z}^{-1}\in L^1_{\loc}(\R^d)$ for $d\geq 2$.
	
	Now let $\phi\in C_c^\infty(\R^d)$ be a smooth bump such that $\phi\equiv 1$ in $Q(0,\frac14)$, and $\phi$ is supported in $Q(0,\frac12)$ (cube of centre $0$ and ``radius'' $\frac12$, hence sidelength $1$). We consider the function $\phi_i(z)=\phi(z)z_i\abs{z}^{d-1}$. By the previous computation and product rule, this satisfies
	$$\abs{\partial^\alpha\phi_i}\lesssim\abs{z}^{-1} 1_{Q(0,\frac12)} \in L^1(\R^d)$$
	for $\abs{\alpha}=d+1$ and $d\geq 2$.
	
	Thus the Fourier transform of $\phi_i$ satisfies for all $\abs{\alpha}=d+1$,
	\begin{equation*}
	\abs{k^\alpha \hat\phi_i(k)}
	\sim
	\Babs{\int \partial_z^\alpha\phi_i(z) e^{-i2\pi k\cdot z}\ud z} \\
	\leq\int\abs{\partial_z^\alpha\phi_i(z)}\ud z<\infty,
	\end{equation*}
	and hence $\abs{\hat\phi_i(k)}\lesssim\abs{k}^{-1-d}$. If $\Phi_i$ is the $1$-periodic extension of $\phi_i$,  its Fourier coefficients satisfy this same estimate.
	In particular, these Fourier coefficient are in $\ell^1(\Z^d)$. Recalling that $\phi_i(z)$ agrees with $z_i\abs{z}^{d-1}$ in $Q(0,\frac14)$, we hence have shown that
	\begin{equation}\label{eq:FourierExp}
	z_i\abs{z}^{d-1}=\sum_{k\in\Z^d} a_i(k)e^{i2\pi k\cdot z},\qquad\forall z\in Q(0,\frac14),
	\end{equation}
	where $\sum_{k\in\Z^d}\abs{a_i(k)}<\infty$.
	
	And observe that we only need to apply the formula \eqref{eq:formula} when $x,y\in Q$, a given cube. By composing back and forth with dilations in addition to translations, we may assume that $Q=Q(0,\tfrac18)$. Then if $x,y\in Q$, we see that $x-y\in Q(0,\tfrac14)$, where \eqref{eq:FourierExp} is valid. Substituting \eqref{eq:FourierExp} with $z=x-y$ into \eqref{eq:formula}, we obtain
	\begin{equation*}
	1=\sum_{i=1}^d\frac{(x_i-y_i)}{\abs{x-y}^{d+1}}\sum_{k\in\Z^d}a_i(k)e^{i2\pi k\cdot(x-y)}
	=\sum_{k\in\Z^d}a_i(k)\sum_{i=1}^d\frac{(x_i-y_i)}{\abs{x-y}^{d+1}}e^{i2\pi k\cdot x} e^{-i2\pi k\cdot y},
	\end{equation*}
	which is a convergent series of expressions of the desired form, namely the Riesz transform kernel multiplied by (bounded) functions that depend only on $x$ or only on $y$. After this, the argument can be concluded in the same way as before.
	
	This Fourier series idea is based on Svante Janson \cite{JANS}.
\end{proof}

We gather two more basic estimates.
\begin{lemma}\label{magiclemma} Let $Q$ be a cube and $b_i\in  L^3_{loc}, i =1,2,$ be such that $\dashint_Qb_i=0$. Then
	\begin{equation}\label{ml1}
	\Bnorm{\avgint{\avgint{(b_1(x)-b_1(y))(b_2(x)-b_2(y))\overline{b_1(x)}\overline{b_2(y)} \ud y \ud x}{Q}}{Q}}{} \geq \avgint{\abs{b_1}^2}{Q}\avgint{\abs{b_2}^2}{Q}
	\end{equation}
	and
	\begin{equation}\label{ml2}
	\avgint{\avgint{(b_1(x)-b_1(y))(b_2(x)-b_2(y))\overline{b_1(x)}\overline{b_2(x)}) \ud y \ud x}{Q}}{Q} = \dashint_Q\abs{b_1b_2}^2 + \Bnorm{\dashint_Qb_1b_2}{}^2,
	\end{equation}
	where we have replaced the latter occurrance of $b_2(y)$ with $b_2(x).$
\end{lemma}
\begin{proof}
	Multiplying out shows that
	\begin{equation*}
	\begin{split}
	&\dashint_Q\dashint_Q(b_1(x)-b_1(y))(b_2(x)-b_2(y))\overline{b_1(x)}\overline{b_2(y)}\ud y \ud x \\
	&= \dashint_Qb_1(x)b_2(x)\overline{b_1(x)}\ud x \dashint_Q\overline{b_2(y)}\ud y - \dashint_Q b_1(x)\overline{b_1(x)}\ud x\dashint_Q b_2(y)\overline{b_2(y)}\ud y\\
	&- \dashint_Q b_2(x)\overline{b_1(x)}\ud x\dashint_Qb_1(y)\overline{b_2(y)}\ud y + \dashint_Q\overline{b_1(x)}\ud x\dashint_Qb_1(y)b_2(y)\overline{b_2(y)}\ud y \\
	&= -\dashint_Q\abs{b_1(x)}^2\ud x \dashint_Q\abs{b_2(y)}^2\ud y - \Bnorm{\dashint_Q b_1(x)\overline{b_2(x)}\ud x}{}^2,
	\end{split}
	\end{equation*}	
	whence
	\begin{equation*}
	\begin{split}
	& \Bnorm{\dashint_Q\dashint_Q(b_1(x)-b_1(y))(b_2(x)-b_2(y)\overline{b_1(x)}\overline{b_2(y)}\ud y \ud x }{}  \\
	&=  \dashint_Q\abs{b_1(x)}^2\ud x \dashint_Q\abs{b_2(y)}^2\ud y + \Bnorm{\dashint_Q b_1(x)\overline{b_2(x)}\ud x}{}^2 \ge  \dashint_Q\abs{b_1(x)}^2\ud x \dashint_Q\abs{b_2(y)}^2\ud y.
	\end{split}
	\end{equation*}
	As for \eqref{ml2} we compute:
	\begin{equation*}
	\begin{split}
	&	\dashint_Q\dashint_Q\left(b_1(x)-b_1(y)\right)\left(b_2(x)-b_2(y)\right)\overline{b_1(x)b_2(x)}\ud y\ud x \\
	&= \dashint_Q\abs{b_1b_2}^2-\dashint_Qb_2\dashint_Q\abs{b_1}^2\overline{b_2}-\dashint_Qb_1\dashint_Q\overline{b_1}\abs{b_2}^2+\dashint_Qb_1b_2\dashint_Q\overline{b_1b_2} \\
	&= \dashint_Q\abs{b_1b_2}^2 + \Bnorm{\dashint_Qb_1b_2}{}^2.
	\end{split}
	\end{equation*}
\end{proof}
The lower bounds now follow by combining lemmas \ref{rieszlemma} and \ref{magiclemma}.
\begin{theorem}\label{lowerbound}
	Let $R_j$, $j= 1,\dots,d,$ be the Riesz transforms, $b_1b_2\in L^2_{loc},$ $b_1,b_2\in L^3_{loc}.$ Then
	\begin{equation*}
	S_2(b_1,b_2)+T_2(b_1,b_2) \leq C_d\sum_{j=1}^{d}\bNorm{C_bR_j}{L^2\to L^2}.
	\end{equation*}
\end{theorem}	
\begin{proof} Denote $\psi_i = b_i-\langle b_i\rangle_Q,$ $i=1,2.$ Then $\int_Q\psi_i$ = 0 and  the assumptions of Lemma \ref{magiclemma} are satisfied by which by \eqref{ml1} and lemma \ref{rieszlemma} we get the necessary condition $S_2$
	\begin{equation*}
	\begin{split}
	\avgint{\abs{\psi_1(x)}^2\ud x}{Q}\avgint{\abs{\psi_2(y)}^2\ud y}{Q} &\leq \Bnorm{\avgint{\avgint{(\psi_1(x)-\psi_1(y))(\psi_2(x)-\psi_2(y))\overline{\psi_2(y)}\overline{\psi_1(x)} \ud y \ud x}{Q} {}}{Q}}{}  \\
	&\leq C_d\sum_{i=1}^{k}\bNorm{C_bR_i}{L^2\to L^2} \left( \avgint{\abs{\psi_1}^2}{Q} \right) ^{1/2}\left( \avgint{\abs{\psi_2}^{2}}{Q} \right) ^{1/2}.
	\end{split}
	\end{equation*}
	For the condition $T_2$, we apply lemma \ref{rieszlemma} with $f_2 = 1_Q,$ $f_1 = \overline{\psi_1\psi_2}$ and lemma \ref{magiclemma} by \eqref{ml2} with $b_i = \psi_i$ to attain
	\begin{equation*}
	\begin{split}
	&\dashint_Q\abs{\psi_1\psi_2}^2 \leq \Bnorm{\dashint_Q\dashint_Q\left(b_1(x)-b_1(y)\right)\left(b_2(x)-b_2(y)\right) \overline{\psi_1(x)\psi_2(x)} 1_Q(y)  \ud y\ud x}{} \\
	&\leq C_d\sum_{i=1}^{k}\bNorm{C_bR_i}{L^p\to L^p} \left( \avgint{\abs{\psi_1\psi_2}^2}{Q} \right) ^{1/2}\left( \avgint{\abs{f_2}^{2}}{Q} \right) ^{1/2} = C_d\sum_{i=1}^{k}\bNorm{C_bR_i}{L^p\to L^p} \left( \avgint{\abs{\psi_1\psi_2}^2}{Q} \right) ^{1/2}.
	\end{split}
	\end{equation*}
	Dividing out equal factors and summing gives the claim.
\end{proof}

\section{Sufficient conditions}
In this section we specify $T$ to be a  Calder\'on-Zygmund operator satisfying the Dini condition. We begin with partially recalling, with only minor modifications, a sparse domination of $T$ from Lerner \cite{L} (see also \cite{LO19})  and its commutators from Ib\'anez-Firnkorn--Rivera-R\'ios  \cite{IR}. The sparse domination would quickly give the boundedness of the commutator $C_bT$ on $L^2,$  whenever $S_p(b_1,b_2)+T_p(b_1,b_2)<\infty$ for any $p>2.$
However, in the last section we find that this is too strong to characterize the boundedness of $C_bT$ on $L^2$ and hence are motivated to introduce the condition $S_{A,B} + T_C$ involving the Young functions $A,B,C,$ that can be made strictly weaker than $S_p+T_p$ for all $p>2.$ Lastly, we prove the upper bound in Theorem \ref{upperbound} with these updated conditions.

We begin with definitions.
\begin{definition}\label{CZO}
	A $d$-dimensional Calder\'on-Zygmund operator $T$ with an $\omega$-Dini -kernel is a $L^2(\R^d)\to L^2(\R^d)$ bounded operator with the representation
	$$
	Tf(x) = \int K(x,y)f(y)\ud y, \qquad x\notin \supp(f),
	$$
	with the kernel $K:\R^d\times\R^d\setminus\{(x,x) : x\in\R^d\}\rightarrow\C$ satisfying the size condition $\abs{K(x,y)} \leq C\abs{x-y}^{-d}$
	and the regularity condition
	\begin{equation*}
	 \abs{K(x,y)-K(x',y)}+\abs{K(y,x)-K(y,x')} \leq \frac{C'}{\abs{x-y}^d}\omega\left(\frac{\abs{x-x'}}{\abs{x-y}}\right),
	\end{equation*}
	whenever $\abs{x-x'}\leq\frac{1}{2}\abs{x-y},$
	with the modulus of continuity $\omega\colon[0,1]\to\R_+$ that is continuous, increasing, subadditive, satisfies $\omega(0)=0$
	and $\Norm{w}{\rm{Dini}} :=
	\int_0^1\omega(t)\frac{dt}{t}<\infty.$
\end{definition}

\begin{definition} Given a $\gamma\in(0,1),$ we say that a collection of sets $\mathcal{F}$ is $\gamma$-sparse, if for all distinct elements $S,R\in\mathcal{F},$
	there exist sets $E_S\subset S,E_R\subset R$ such that $E_S\cap E_R=\emptyset$ and $\abs{E_S}>\gamma\abs{S}.$
\end{definition}

\begin{definition}  Let $T$ be as in definition \ref{CZO}. We have the following maximal operators on $L^2(\R^d):$
	\begin{enumerate} [i)]
	\item the maximal operator  $T_*f(x)=\sup_{\varepsilon>0}\bnorm{Tf1_{B(x,\varepsilon)^c}(x)}{},$
	\item the grand maximal operator  $\mathcal{M}_T(f)(x) = \sup_{Q\ni x}\esssup_{\xi\in Q}\bnorm{T(f1_{\R^d\setminus3Q})(\xi)}{},$
	\item and its localized version  $\mathcal{M}_{T,Q}(f)(x) = \sup_{Q\supset P\ni x}\esssup_{\xi\in P}\bnorm{T(f1_{3Q\setminus3P})(\xi)}{},$
		where $Q,P$ are cubes.
	\end{enumerate}
\end{definition}
The control over the grand maximal operator is given by
\begin{lemma}\cite[Lemma 3.2]{L}\label{keylemma} Let $f\in L^2_{loc}.$ The following pointwise estimates hold:
	\begin{enumerate}[i)]
	\item for a.e. $x\in Q$ we have:  $\bnorm{T(f1_{3Q}(x))}{} \leq C_d\bNorm{T}{L^1\to L^{1,\infty}}\abs{f(x)}+ \mathcal{M}_{T,Q}f(x),$
	\item for all $x\in\R^d$ we have: $\mathcal{M}_Tf(x) \leq C_d(\bNorm{\omega}{{\rm{Dini}}}+C_T) \mathcal{M}f(x)+T_*f(x).$
	\end{enumerate}
\end{lemma}
For a more refined argument for the sparse domination in Theorem \ref{sparsedom} without Lemma \ref{keylemma}, see the latest version of the sparse domination principle in Lerner, Ombrosi \cite{LO19}.

\begin{theorem}\label{sparsedom} Let $T$ be a d-dimensional Calder\'on-Zygmund operator with a Dini kernel and denote $b(x,y) = (b_1(x)-b_1(y)) (b_2(x)-b_2(y)).$ We assume that $f\in L^1_c(\R^d),$ and further to make everything well-defined that $b_1,b_2,b_1b_2,b_1f,b_2f,b_1b_2f\in L^1_{loc}.$

	From these assumptions it follows that there exists a sparse collection $S$ of cubes on $\R^d$ such that
	\begin{equation*}
		\Bnorm{C_bTf(x)}{} \leq C_{T,d} \sum_{i=1}^{4}S_if(x),
	\end{equation*}
	where
	\begin{equation*}
	S_1f = \sum_{Q\in S} \abs{b_1-\langle b_1\rangle_Q} \abs{b_2-\langle b_2\rangle_Q}\big\langle\abs{f}\big\rangle_Q1_Q, \qquad
	S_2f = \sum_{Q\in S} \abs{b_2-\langle b_2\rangle_Q}\big\langle\abs{b_1-\langle b_1\rangle_Q}\abs{f}\big\rangle_Q1_Q,
	\end{equation*}
	\begin{equation*}
	S_3f = \sum_{Q\in S} \abs{b_1-\langle b_1\rangle_Q}\big\langle\abs{b_2-\langle b_2\rangle_Q}\abs{f}\big\rangle_Q1_Q,\qquad
	S_4f = \sum_{Q\in S} \big\langle\abs{b_1-\langle b_1\rangle_Q}\abs{b_2-\langle b_2\rangle_Q}\abs{f}\big\rangle_Q1_Q,
	\end{equation*}
	and the sparse constant denoted with $\gamma$ depends only on the dimension $d.$
\end{theorem}

\begin{proof}
	We recall only the part of the proof where the exceptional set is defined and control over the appearing terms is established. In addition, a comment is made about the rest of the proof, the details for which we refer the reader to the proof of Theorem 1.1 in \cite{LOR1} or \cite{L}.
	
	For an arbitrary integrable function $\psi\not=0$ on $Q$ define
	\begin{equation*}
	E_1(\psi) = \{ x\in Q\colon \abs{\psi(x)} > \alpha \langle\abs{\psi}\rangle_{3Q} \},\qquad
	E_2(\psi) = \{x\in Q\colon \mathcal{M}_{T,Q}\psi(x) > \alpha \langle\abs{\psi}\rangle_{3Q} \}
	\end{equation*}
	and let the exceptional set be
	\begin{equation*}
	E = \bigcup_{i=1,2} E_i(f)\cup E_i(b_1f) \cup E_i(b_2f)\cup E_i(b_1b_2f).
	\end{equation*}
	Since the localized version of the grand maximal operator is controlled with the non-localized by
	\begin{equation*}
	\mathcal{M}_{T,Q}f \leq \mathcal{M}_T(f1_{3Q}),
	\end{equation*}
	and by the well-known facts that $\mathcal{M},T_*\colon L^1\to L^{1,\infty}$ boundedly,
 	it follows from the weak $(1,1)$ bounds implied by ii) of Lemma \ref{keylemma} in conjunction with the local integrability of all functions in question that we may choose some $\alpha>0$ independent of the cube $Q$ so that $\abs{E}\leq 2^{-(d+2)}\abs{Q}$.

	Taking a Calder\'on-Zygmund decomposition of the function $1_E$ at the height $2^{-(d+1)}$ yields a collection $\mathcal{F}$ of cubes satisfying:
\begin{equation*}
		\sum_{P\in\mathcal{F}}\abs{P} \leq\frac{1}{2}\abs{Q}, \quad\abs{E\setminus\bigcup_{P\in\mathcal{F}}P}=0\quad\mbox{ and }\quad P\cap E^c\not=\emptyset\qquad\forall P\in\mathcal{F}.
\end{equation*}
	Then one decomposes
	\begin{equation*}
		\left( C_bT(f1_{3Q})\right)1_Q = 	\left( C_bT(f1_{3Q})\right)1_{Q\setminus\cup P} + \sum_{P\in\mathcal{F}}\left( C_bT(f1_{3Q\setminus 3P})\right)1_P + \sum_{P\in\mathcal{F}}\left( C_bT(f1_{3P})\right)1_P
	\end{equation*}
	and uses the properties of the collection $\mathcal{F},$ Lemma \ref{keylemma} and that the commutator is unchanged modulo constants in the functions $b_1,b_2$ to derive
	\begin{equation*}
	\begin{split}
		\Bnorm{C_bT(f1_{3Q})}{} 1_Q&\leq C_{T,d}\Big( \abs{b_2-\langle b_2\rangle_{3Q}}\abs{b_1-\langle b_1\rangle_{3Q}}\langle\abs{f}\rangle_{3Q} \\
		&+ \abs{b_2-\langle b_2\rangle_{3Q}}\langle\abs{b_1-\langle b_1\rangle_{3Q}}\abs{f}\rangle_{3Q} \\
		&+ \abs{b_1-\langle b_1\rangle_{3Q}}\langle\abs{b_2-\langle b_2\rangle_{3Q}}\abs{f}\rangle_{3Q} \\
		&+ \langle\abs{b_1-\langle b_1\rangle_{3Q}}\abs{b_2-\langle b_2\rangle_{3Q}}\abs{f}\rangle_{3Q} \Big)1_Q + \sum_{P\in\mathcal{F}}\Bnorm{C_bT(f1_{3P})}{}1_P .
			\end{split}
	\end{equation*}
	From this situation one first iterates the above estimate with the last term and then transfers the limit construction from the local to the global.
\end{proof}

Before stating and proving theorem \ref{upperbound} we need to recall and define
\subsection{Young functions, their basic properties and the conditions $S_{A,B},T_C$}

We may also define joint conditions involving Young functions.
A function $A:[0, \infty)\rightarrow [0, \infty)$ is called  a Young function if it is continuous, convex, strictly increasing and satisfies
$$A(0)=0,\qquad\lim_{t\rightarrow \infty} A(t)/t =\infty.$$
Given a Young function $A$, the complementary Young function $\bar A$ is defined by
\[
\bar A(t)= \sup_{s>0} \{st-A(s)\}, \quad t>0.
\]
We also have the maximal function associated with a Young function $A:$
\[
M_Af(x)= \sup_{Q\ni x}\langle |f| \rangle_{A, Q},
\] where the Luxemburg norm is defined by \[
\langle |f |\rangle_{A, Q}= \inf\{\lambda>0: \frac 1{|Q|}\int_Q A(|f|/\lambda)\le 1\}.
\]

We say that $f\in L^A_{\rm{loc}}$ if $\langle |f |\rangle_{A, Q}<\infty$ for all cubes $Q$.
The relative sizes of Young functions $A,B$ are compared with the  symbol $\succeq;$ we say that $B\succeq A,$ if there exist constants $C,t_0>0$ such that $A(t)\leq CB(t),$ when $t>t_0.$
Finally, we define the $B_p$ class: a Young function $A\in B_p$ for $p>1$ if
\[
\int_1^\infty \frac{A(t)}{t^p}\frac{dt}{t}<\infty.
\]
We record the following properties, which can be found at least in \cite[Chapter 5]{CMP10} (see also \cite{P95}):
\begin{proposition}\label{prop:young}
	Given a Young function $A$, it holds that
	\begin{enumerate}[i)]
		\item for any $t\ge 0$, $t \le A^{-1}(t) \bar A^{-1}(t)\le 2t$,
		\item for any cube $Q$,
		\begin{equation}\label{eq:generalholder}
		\langle |f g| \rangle_Q\le 2\langle |f| \rangle_{A, Q} \langle |g| \rangle_{\bar A, Q}.
		\end{equation} More generally, if $A$, $B$, and $C$ are Young functions such that for all $	t\ge t_0>0,$
		\[
		B^{-1}(t) C^{-1}(t)\le c A^{-1}(t),
		\]
		then
		\[
		\langle |fg|\rangle_{A,Q}\lesssim \langle |f|\rangle_{B,Q}\langle |g|\rangle_{C,Q},
		\]
		\item if $B\succeq A,$ then $\langle\abs{f}\rangle_{A,Q}\lesssim \langle\abs{f}\rangle_{B,Q}$ and $M_A \lesssim M_B,$
		\item if $\bar A\in B_{p'}$, then $A(t)\succeq t^p$ and $ \langle |f|^p\rangle_Q^{\frac 1p}\lesssim \langle|f|\rangle_{A, Q}$.
	\end{enumerate}
\end{proposition}
\begin{proposition}\cite{P95}\label{PE95} $M_A:L^p\to L^p$ boundedly if and only if $A\in B_p$.
\end{proposition}

Now we are ready to give the following definition:

\begin{definition}\label{sufdef}	
	Given Young functions $A, B, C$ such that $\bar B,\bar C\in B_p,$ $ \bar A\in B_{p'}$ and a pair of complex valued functions $b_1\in L^A_{\rm{loc}}(\R^d), b_2\in L^B_{\rm{loc}}(\R^d)$, we say that the joint condition $S_{A,B}$ holds if
	\begin{align*}
	S_{A,B}(b_1, b_2):= \sup_Q \bla |b_1- \langle b_1\rangle_Q|\bra_{ A, Q}\bla |b_2- \langle b_2\rangle_Q|\bra_{ B, Q}<\infty,
	\end{align*}
	and for $b_1^2, b_2^2\in L^{C}_{\rm{loc}}(\R^d)$, we say that the joint condition $T_{C}$ holds if
	\begin{equation}\label{ms}
	T_{C}(b_1, b_2):=  \sup_Q \bla |b_1- \langle b_1\rangle_Q| |b_2- \langle b_2\rangle_Q| \bra_{ C, Q}<\infty.
	\end{equation}
\end{definition}

We remark immediately, that in Theorem \ref{jnce-1} we find a commutator that is unbounded on $L^2$ and that satisfies the conditions $S_2+T_2$ but fails the conditions $S_{A,B} + T_C$ for all Young functions $\bar A,\bar B,\bar C\in B_2.$

\begin{theorem}\label{upperbound} Assume that a pair of functions $b_1\in L^A_{\rm{loc}}(\R^d)$ and $b_2\in L^B_{\rm{loc}}(\R^d)$ with $b_1^2, b_2^2\in  L^{C}_{\rm{loc}}(\R^d)$ satisfy  the conditions $T_C$ and $S_{A,B}$ for some Young functions $A,B,C$ with $\bar A, \bar B, \bar C\in B_2$, then
	\begin{equation*}
	\begin{split}
	S_i\colon L^2(\R^d)\cap L^3_c(\R^d) \longrightarrow L^2(\R^d), \qquad i=1,2,3,4
	\end{split}
	\end{equation*}
	boundedly.
	
	Especially, it follows with a standard density argument by Theorem \ref{sparsedom}  that
	\begin{equation*}
	\begin{split}
	C_bT\colon L^2(\R^d) \longrightarrow L^2(\R^d)
	\end{split}
	\end{equation*}
	boundedly when notation and assumptions are retained.
\end{theorem}
	\begin{proof}  The pairs of terms $S_1,S_4$ and $S_2,S_3$ are symmetric with respect to dual pairings. Hence, we show the estimate in the two distinct cases of $S_1$ and $S_3.$ By duality it is enough to estimate the pairings  $\langle S_i(f),\psi\rangle$.
		
		First, for the term $S_1$ we only use the assumptions involving the functions $A,B.$  By sparseness we get
		\begin{equation*}
		\begin{split}
		\Bnorm{\langle S_1(f),\psi\rangle}{} &\leq \sum_{Q\in S} \int_Q\abs{b_1-\langle b_1\rangle_Q}\abs{\psi}\langle\abs{b_2-\langle b_2\rangle_Q}\abs{f}\rangle_Q \\
		&\lesssim \sum_{Q\in S} \abs{Q}\langle \abs{b_1-\langle b_1\rangle}\rangle_{A,Q}\langle\abs{\psi} \rangle_{\bar A, Q}\langle\abs{b_2-\langle b_2\rangle_Q}\rangle_{B,Q}\langle\abs{f} \rangle_{\bar B, Q} \\
		&\leq S_{A,B}(b_1,b_2) \sum_{Q\in S} \abs{Q}\langle\abs{\psi} \rangle_{\bar A,Q}\langle\abs{f} \rangle_{\bar B, Q} \leq \gamma^{-1}S_{A,B}(b_1,b_2) \sum_{Q\in S} \int_{E_Q}\langle\abs{\psi} \rangle_{\bar A,Q}\langle\abs{f} \rangle_{\bar B, Q} \\
		&\leq \gamma^{-1}S_{A,B}(b_1,b_2) \sum_{Q\in S} \int_{E_Q}  {M}_{\bar A}\psi {M}_{\bar B}f\leq  \gamma^{-1}S_{A,B}(b_1,b_2) \bNorm{ {M}_{\bar A}\psi}{L^2}\bNorm{ {M}_{\bar B}f}{L^2} \\
		&\lesssim S_{A,B}(b_1,b_2)  \bNorm{\psi}{L^2}\bNorm{f}{L^2},
		\end{split}
		\end{equation*}
where we have used Proposition \ref{PE95} in the last step.

Next we use the condition $T_C$ to control the term $S_3$:
\begin{align*}
\Bnorm{\langle S_3(f),\psi\rangle}{} &\leq \sum_{Q\in S} \abs{Q}\langle \abs{\psi} \rangle_Q  \langle\abs{b_1-\langle b_1\rangle_Q}\abs{b_2-\langle b_2\rangle_Q}\abs{f}\rangle_Q \\
&\leq  \sum_{Q\in S} \abs{Q}\langle \abs{\psi} \rangle_Q \langle \abs{f} \rangle_{\bar C, Q} \langle\abs{b_1-\langle b_1\rangle_Q} \abs{b_2-\langle b_2\rangle_Q} \rangle_{C,Q} \\
&\leq T_C(b_1,b_2)\sum_{Q\in S} \abs{Q}\langle \abs{\psi} \rangle_Q \langle \abs{f} \rangle_{\bar C, Q}  \leq \gamma^{-1} T_C(b_1,b_2)\bNorm{M\psi}{L^{2}}\bNorm{M_{\bar C}f}{L^2}\\
&\lesssim  T_C(b_1,b_2)\bNorm{\psi}{L^{2}}\bNorm{f}{L^2}.
\end{align*}
 	\end{proof}

Since with $A(t)=t^p,$ $\bar A \in B_{2},$ for $p>2,$ we immediately get:
\begin{corollary}\label{rmk:p}
	Let $T$ be as before and assume that a pair of functions $b_1,b_2\in L^{2p}_{loc}(\R^d)$ satisfy the conditions $T_{p}$ and $S_{p}$ for some $p>2.$ Then we have
	\begin{equation*}
	\begin{split}
	C_bT\colon L^2(\R^d)\longrightarrow L^2(\R^d)
	\end{split}
	\end{equation*}
	boundedly.
\end{corollary}

We close this section with some remarks.

\begin{remark}
For Theorem \ref{upperbound} the difference in the case $p\not=2$ is that we need to introduce 3 more Young functions to manage the now non-symmetric dual pairings from the terms $S_2,S_4.$  According to Definition \ref{sufdef} the existing Young functions functions are replaced with ones satisfying
\begin{equation*}
\bar A \in B_{p'},\quad \bar B,\bar C\in B_p
\end{equation*}
and are supplemented with Young functions $D,E,F$ satisfying
\begin{equation*}
\bar D,\bar F\in B_p,\quad \bar E\in B_{p'}
\end{equation*}
and $S_{D,E}(b_1,b_2)+T_F(b_1,b_2) < \infty.$
\end{remark}

\begin{remark}\label{altsuf} Given a $q\in(2,\infty),$ adapting the proof of Theorem \ref{upperbound} shows that if $b_1,b_2$ satisfy the conditions $S_{q+\varepsilon},T_{q+\varepsilon}$ for any $\varepsilon>0,$ then $C_bT:L^q\to L^q$ boundedly.

On the other hand, for $q\in(1,2),$ the conditions $S_p,T_p$ with $p\in(q,2)$ are not strong enough to conclude  that $C_bT:L^q\to L^q$ boundedly. Indeed, if they were, then by duality and interpolation $C_bT:L^2\to L^2$ boundedly and Theorem \ref{lowerbound} would imply the condition $S_2$. This gives a contradiction since by Proposition \ref{prop1} (see below) there exist functions $\phi,\psi$ such that $S_p,T_p$ are satisfied and $S_2$ is not.

\end{remark}

\section{Conjecture and related examples}
In this last section we continue discussing the conditions $S_{A,B}, T_C$ and their interdependence with the boundedness properties of the commutator on different $L^p$ spaces.

First, we note that it follows by the John-Nirenberg inequality that if $b_1,b_2\in\BMO,$ then the conditions $S_p,T_q$ hold for all $p,q\geq1.$
Hence, a natural question is immediate: Are $S_p$, and respectively $T_p$, equivalent for all or some $1\leq p<\infty.$
Or even in a weaker sense: if both of the conditions $S_p,T_p$ hold simultaneously, could we deduce that $S_q$ or $T_q$ holds for some $q>p?$ By Theorem \ref{nip1} the answer is no and the example located therein is the motivation for introducing joint conditions involving Young functions that can be made strictly weaker than $S_{2+\varepsilon}+T_{2+\varepsilon}$ for all $\varepsilon >0.$

The next proposition will clarify the situation and point out how the counterexample in Theorem \ref{nip1} can be constructed. For this, recall, that a function $\omega:\R^d\to(0,\infty)$ is said to be in the class of $A_p$ weights, $1<p<\infty$, if
$$[w]_{A_p} = \sup_{Q} \ave{w}_Q\ave{w^{-\frac{p'}{p}}}_Q^{\frac{p}{p'}}<\infty,\quad p' = \frac{p}{p-1},$$
where the supremum is taken over all cubes.

\begin{proposition}\label{prop1} Given $1<p<q<\infty,$ there exists functions $\phi,\psi\in L^p_{loc}$ satisfying the conditions $S_p,T_p$ and failing the condition $S_q.$
\end{proposition}
\begin{proof} Let
	\begin{equation*}
	\psi(x) = x^{-\frac{2}{p+q}}1_{(0,1)},\qquad \phi = \psi^{-1}1_{(0,1)}.
	\end{equation*}
	We check that the conditions $S_p,T_p$ hold.
	Let $[a,b)$ be an arbitrary intervall such that $[a,b)\cap[0,1)=[c,d)\not=\emptyset$ (if the intersection is empty, then the claim is trivial). First,
	\[
	\frac{1}{b-a}\int_a^b\abs{\psi-\langle\psi\rangle_{[a,b)}}^p\frac{1}{b-a}\int_a^b\abs{\phi-\langle\phi\rangle_{[a,b)}}^p \leq 	\frac{1}{d-c}2^{p+1}\int_c^d\psi^p\frac{1}{d-c}2^{p+1}\int_c^d\phi^p.
	\]
	Then, by the fact (see Grafakos \cite{CFA}) that $|x|^{-\frac{2p}{p+q}}\in A_2$, we have
	\[
	\frac{1}{d-c} \int_c^d\psi^p\frac{1}{d-c} \int_c^d\phi^p\le [|x|^{-\frac{2p}{p+q}}]_{A_2}.
	\]	
	It follows that $S_p(\psi,\phi) \lesssim 1.$
	
	By the above estimates and $\phi\psi \leq1,$ it follows for an arbitrary interval $I$ that
	\[
	\frac{1}{\abs{I}}\int_I\abs{\psi-\langle\psi\rangle_I}^p\abs{\phi-\langle\phi\rangle_I}^p \leq 4^p	\frac{1}{\abs{I}}\left(\int_I\psi^p\phi^p + \int_I\psi^p\langle\phi\rangle_I^p+ \int_I\phi^p\langle\psi\rangle_I^p + \langle\psi\rangle_I^p \langle\phi\rangle_I^p      \right) \lesssim 1.
	\]
	Hence $T_p(b_1,b_2)<\infty.$
	
	On the other hand by $-2q/(p+q) < -1,$ the singularity in $\int_0^1\abs{\psi-\langle\psi\rangle_{[0,1)}}^q$ is not integrable, and by $\int_0^1\abs{\phi-\langle\phi\rangle_{[0,1)}}^q > 0,$ we have $S_q(\psi,\phi)=\infty.$
\end{proof}

\begin{remark} If one wishes to have $\psi,\phi\in L^{\infty}_{loc},$ say to have the joint conditions well-defined, Proposition \eqref{prop1} can be modified by considering multiple copies of the situation spread out through $\R$ and introducing the singularities in $\psi$'s only gradually as is done in the next theorem.
\end{remark}

\begin{theorem}\label{nip1} There exist functions $\psi,\phi \in L^{\infty}_{loc}$ failing the condition $S_{2+\varepsilon}$ for all $\varepsilon > 0,$ such that  $[\phi,[\psi,H]]:L^2\to L^2$ boundedly, where $H$ is the Hilbert transform, i.e. the 1-dimensional Riesz transform.

\end{theorem}
\begin{remark}
	By Theorem \ref{lowerbound} the $L^2$ boundedness implies that $\psi,\phi$ satisfy the conditions $T_2,S_2.$
\end{remark}
\begin{proof}[Proof of Theorem \ref{nip1}]
	Let
	\begin{equation*}
	\psi_0^k(x) = c_kx^{-\eta_k}1_{(c_k^{6k}e^{-100k^2},1)}(x), \qquad \eta_k = \frac{1}{2+k^{-1}},\qquad
	\phi_0(x) = x^{1/2}1_{(0,1)}(x),
	\end{equation*}
	where $c_k$ depends on $k$ and will be determined later.
	Let $\tau_hf(x) = f(x-h).$
	Then set $\phi_k = \tau_k\phi_0$  and $\psi_k = \tau_k\psi^{k}_0$.
	Finally we define
	\begin{equation*}
	\phi = \sum_{k\in 2\mathbb Z }\phi_k,\qquad \psi = \sum_{k\in 4\mathbb N +2}\psi_k.
	\end{equation*}
	Let $k\in 4\mathbb N +2$ be fixed. We first show that the pair $(\psi_k, \phi)$ satisfies $S_q,T_q$ for $q=q_k=2+k^{-1}/2.$
	Since $\tau_{k}\phi=\phi$ it suffices to prove that $(\psi_0^k, \phi)$ satisfies $S_q,T_q$. Again, for any interval $I$, we have
	\begin{align*}
	\frac 1{|I|}\int_I \abs{\psi_0^k-\langle\psi_0^k\rangle_{I} }^q\frac 1{|I|}\int_I \abs{\phi -\langle\phi \rangle_{I} }^q\le 4^{q+1} \langle |\psi_0^k|^q\rangle_I \langle |\phi |^q\rangle_I.
	\end{align*} We first consider the case when $\ell(I)\le 1$ and we may further assume that $I\subset (0,1)$.
	Since $q\eta_k<1$ we know that $|x|^{-q\eta_k}\in A_2$ and hence, by $I\subset(0,1),$
	$$
	4^{q+1} \langle |\psi_0^k|^q\rangle_I \langle |\phi |^q\rangle_I\le 4^{7/2} c_k^{5/2}[|x|^{-q\eta_k}]_{A_2}.
	$$It remains to consider the case when $\ell(I)>1$. Since certainly $(0,1)\cap I\neq \emptyset$ (as otherwise there is nothing to prove) we know that $(0,1)\subset 3I$. Then due to that $\phi$ is a periodic function we have
	\[
	4^{q+1} \langle |\psi_0^k|^q\rangle_I \langle |\phi |^q\rangle_I\le  4^{q+1} \langle |\psi_0^k|^q\rangle_{(0,1)} \langle |\phi |^q\rangle_{(0,1)}\le  4^{7/2} c_k^{5/2}[|x|^{-q\eta_k}]_{A_2}.
	\]Therefore, we conclude that
	\[
	S_q\le  4^{7/2} c_k^{5/2}[|x|^{-q\eta_k}]_{A_2}.
	\]
	On the other hand, since $\psi_0^k \phi_0\le c_k$, then by similar arguments as in Proposition \ref{prop1} we have
	$$
	T_q\le 4^{7/2} c_k^{5/2}[|x|^{-q\eta_k}]_{A_2}.
	$$
	Hence by Theorem \ref{upperbound} (see below) we know that the commutator $[\phi,[\psi_k,H]]$ is bounded  on $L^2$ with norm $\sim c_k^{5/2} [|x|^{-q\eta_k}]_{A_2}\|M_{q'}\|_{L^2\rightarrow L^2}^2$. Thus, we may further demand  the constant $c_{k}$ to be so small that $\bNorm{[\phi,[\psi_k,H]]}{L^2\to L^2} \leq 2^{-k}$. Then
	$[\phi,[\psi, H]]$ also is bounded on $L^2$:
	\begin{equation*}
	\bNorm{	[\phi,[\psi,H]]}{L^2\to L^2} = \bNorm{[\phi,[\sum_{k\in 4\mathbb N+2}\psi_k,H]]}{L^2\to L^2} \leq \sum_{k\in4\mathbb N+2}\bNorm{[\phi,[\psi_k,H]]}{L^2\to L^2} \leq \sum_{k=1}^{\infty}2^{-k}=1.
	\end{equation*}		
	
	It remains to check that the pair $(\psi,\phi)$ is precisely what we need. It is obvious that $\psi,\phi\in L_{loc}^\infty$.  It remains to verify that $(\psi,\phi)$ fails $S_{2+\varepsilon}$ for any $\varepsilon>0$.
	By H\"older's inequality we can assume  $0<\varepsilon<1.$ Find $\ell\in \N$ such that with $k:= 4l+2$ it holds that $(2+\varepsilon)\eta_k >1+(2k)^{-1}.$ Hence, with $I=(k,k+1)$ we get
	\begin{align*}
	\int_{I}\Babs{\psi-\langle\psi\rangle_{I}}^{2+\varepsilon} = \int_{0}^1\Babs{\psi_0^k-\langle\psi_0^k\rangle_{(0,1)}}^{2+\varepsilon} &\ge  c_k^{2+\varepsilon} \int_{c_k^{6k} e^{-100k^2}}^{2c_k^{6k} e^{-100k^2}}\Babs{ x^{-\frac{1}{2+k^{-1}}} - \int_{c_k^{6k}e^{-100k^2}}^1 x^{-\frac{1}{2+k^{-1}}} }^{2+\varepsilon}\ud x \\
	&\gtrsim c_k^{2+\varepsilon} \int_{c_k^{6k} e^{-100k^2}}^{2c_k^{6k} e^{-100k^2}}\Babs{ x^{-\frac{1}{2+k^{-1}}}   }^{2+\varepsilon}\ud x\\
	&\gtrsim c_k^{\varepsilon-1} e^{50k}.
	\end{align*}
	On the other hand,
	\begin{equation*}
	\int_{I}\Babs{\phi-\langle\phi\rangle_{I}}^{2+\varepsilon} = \int_{0}^1\Babs{\phi_0-\langle\phi_0\rangle_{(0,1)}}^{2+\varepsilon}\sim 1.
	\end{equation*}
	We conclude the proof by letting $\ell \rightarrow \infty$.
\end{proof}

\begin{theorem}\label{jnce-1}
	There exists $b_1, b_2\in L^{\infty}_{\rm{loc}}(\R)$ such that $S_2(b_1, b_2)+T_2(b_1, b_2)<\infty$, but $S_{A,B}(b_1, b_2)=\infty$ and $T_C(b_1, b_2)=\infty$ for arbitrary Young functions $A,B,C$ with $\bar A, \bar B, \bar C\in B_2$.
	Moreover, $C_bH: L^2\not\to L^2.$
\end{theorem}

\begin{proof} We prove the result via the following example.
	Let $I_0=[-1,1]$ and
	\[
	\sigma= 1_{I_0},\qquad w=M(\sigma)^{-1},
	\]
	notice that both $\sigma$ and $w$ are even functions. It is immediate to see that
	\begin{equation}\label{eee}
	\sup_I \langle \sigma\rangle_I \langle w \rangle_I \le \sup_I \inf_{x\in I} M(\sigma)(x) \langle w \rangle_I\le \sup_I\langle  M(\sigma)w \rangle_I =1.
	\end{equation}
	Now define
	\begin{align*}
	b_1(x):= \sgn(x)\sigma(x),\qquad b_2(x):= \sgn(x) w^{\frac 12}(x),
	\end{align*}
	and notice that immediately $b_1, b_2\in L_{\rm{loc}}^{\infty}.$
	By \eqref{eee} we see that
	\begin{equation}\label{end1}
	S_2(b_1, b_2)^2=\sup_I \bla |b_1-\langle b_1\rangle_I|^2\bra_I \bla |b_2-\langle b_2\rangle_I|^2\bra_I\le 16 \sup_I \bla |b_1|^2\bra_I \bla |b_2|^2\bra_I \le 16< \infty.
	\end{equation}
	We also have
	\begin{align*}
	|b_1-\langle b_1\rangle_I|^2|b_2-\langle b_2\rangle_I|^2 & \le 4(|b_1|^2+ |\langle b_1\rangle_I|^2)(|b_2|^2+ |\langle b_2\rangle_I|^2),
	\end{align*}
	and by $|b_1b_2|\le 1$, direct calculations give us
	\begin{equation}\label{end2}
	T_2(b_1, b_2)^2=\sup_I \bla |b_1-\langle b_1\rangle_I|^2|b_2-\langle b_2\rangle_I|^2\bra_I \le 4+ 12\bla |b_1|^2\bra_I \bla |b_2|^2\bra_I\le 16<\infty.
	\end{equation}
	However, for $J_k=(-k,k)$, $k\ge 2$, since $b_1$ and $b_2$ are odd functions,
	\begin{align*}
	S_{A, B}(b_1, b_2) &\ge \lim_{k\rightarrow \infty}\bla |b_1 |\bra_{ A, J_k}\bla |b_2 |  \bra_{B, J_k} \gtrsim \lim_{k\rightarrow \infty}\bla |b_1 |\bra_{ A, J_k}\bla |b_2 |^2  \bra_{ J_k}^{\frac 12} \\
	&\sim \lim_{k\rightarrow \infty}\frac{k^{\frac 12}}{ A^{-1}(k)}\sim \lim_{k\rightarrow \infty}\frac{\bar A^{-1}(k)}{k^{\frac 12}}= \lim_{k\rightarrow \infty}\Big(\frac{(\bar A^{-1}(k))^2}{\bar A(\bar A^{-1}k) }\Big)^{\frac12},
	\end{align*}
	where we have used the fact that $M(1_{I_0})(x)\sim (1+|x|)^{-1}$. To conclude notice that immediately by definition $\lim_{t\rightarrow \infty}\bar A^{-1}(t)=\infty$  and
	\[
	\lim_{t\rightarrow \infty} \frac{\bar A(t)}{t^2}\le \frac 4{\log 2}\lim_{t\rightarrow \infty} \int_{t}^{2t} \frac{\bar A(s)}{s^2}\frac{ds}{s}=0.
	\]
	On the other hand with $I_k=(0, k)$, $k\ge 100,$ we have
	\begin{align*}
	T_{C}(b_1, b_2)&\ge \lim_{k\rightarrow \infty}  \bla |b_1- \langle b_1\rangle_{I_k}| |b_2- \langle b_2\rangle_{I_k}| \bra_{ C, I_k} \ge \lim_{k\rightarrow \infty}  \bla |b_1- \langle b_1\rangle_{I_k}| |b_2- \langle b_2\rangle_{I_k}| 1_{(0,1)}\bra_{ C, I_k}\\
	&=\lim_{k\rightarrow \infty} (1-k^{-1})  \bla   |b_2- \langle b_2\rangle_{I_k}| 1_{(0,1)}\bra_{ C, I_k}.
	\end{align*}
	Since for $x>1$, $b_2(x)=(\frac{x+1}{2})^{\frac 12}$ and for $0<x\le 1$, $b_2(x)=1,$ another direct calculation shows that
	\[
	\langle b_2\rangle_{I_k}= \frac{\sqrt 2}{3k}\big[ (k+1)^{\frac 32}-2\sqrt 2 \big]
	\]
	by which and the assumption $k\ge 100$ we see that
	\[
	|b_2- \langle b_2\rangle_{I_k}| 1_{(0,1)}\ge c k^{\frac 12}.
	\]
	Hence
	\begin{align*}
	T_{C, 2}(b_1, b_2)\ge\lim_{k\rightarrow \infty}  c(1-k^{-1})k^{\frac 12}\langle 1_{(0,1)}\rangle_{C, I_k}=\lim_{k\rightarrow \infty}  c(1-k^{-1})\frac{k^{\frac 12}}{C^{-1}(k)}=\infty.
	\end{align*}
	Next, we show that  $C_bH: L^2\not\to L^2.$
	To see this, let
	\begin{equation*}
		f(x)=x^{-\frac 12}(\log x)^{-1}1_{[100,\infty)}(x) \in L^2(\R).
	\end{equation*}
	We claim that $|C_bHf(x)|=\infty$ for all $x\in I_0$, and in showing this, hence conclude the unboundedness of $C_bH$ on $L^2$. Indeed, since for $y\in [100,\infty)$
	\[
	b_2(y)=(M(1_{I_0}))^{-\frac 12}=\sqrt{ \frac {y+1}2},
	\] and $b_1(x)=b_2(x)=\sgn(x)$, for any $x\in I_0,$ we have
	\begin{align*}
	|C_bHf(x)|&=\Big| \int_{100}^\infty (b_1(x)-b_1(y)) (b_2(x)-b_2(y))\frac{f(y)}{x-y}dy\Big| =\Big| \int_{100}^\infty   (\sgn(x)-\sqrt{\frac{y+1}2})\frac{f(y)}{x-y}dy\Big|\\
	&\sim\int_{100}^\infty   \sqrt{\frac{y+1}2}\frac{f(y)}{y-x}dy\sim \int_{100}^\infty \frac {f(y)}{\sqrt y}dy=\infty.
		\end{align*}
	
\end{proof}

If we take $A(t)=B(t)=C(t)=t^{2+\varepsilon}$, where $\varepsilon>0$, we immediately have the following
\begin{corollary}\label{JNCEE} The conditions $S_2,T_2$ holding simultaneously does not improve to $S_{2+\varepsilon}$ or $T_{2+\varepsilon}$ for any $\varepsilon > 0$.
\end{corollary}

For our next example, we note that functions $\Phi:[0,\infty)\to[0,\infty)$ of the form
\[
\Phi(t) = t^p\log(e+t)^{p-1+\delta},\qquad p\in (1,\infty), \ \delta\in (0,\infty)
\]
are called log -bumps.
These are Young functions, and we recall some facts from \cite[Chapter 5]{CMP10}:
\begin{enumerate}[i)]
\item If $\Phi(t) = t^p\log(e+t)^{p-1+\delta},$ then $$\Phi^{-1}(t) \sim t^{1/p}\log(e+t)^{-\frac{1}{p'}-\frac \delta p} \mbox{ and } \bar \Phi(t) \sim t^{p'}[\log(e+t)]^{-1-(p'-1)\delta}\in B_{p'}.$$
    \item If $\Phi(t) = t^p\log(e+t)^{p-1}\log\log(e^e+t)^{p-1+\delta}$(which is referred to as a loglog-bump),
    then $$\Phi^{-1}(t) \sim t^{\frac 1 p}\log(e+t)^{-\frac1{p'}}\log\log(e^e+t)^{-\frac{1}{p'}-\frac \delta p}$$
    and
    $$\bar \Phi(t) \sim t^{p'}\log(e+t)^{-1}[\log\log(e^e+t)]^{-1-(p'-1)\delta}\in B_{p'}.$$
\end{enumerate}

\begin{theorem}\label{lastexample} There exist functions $b_1,b_2 \in L^{\infty}_{loc}$ such that $C_bH:L^2\to L^2$ boundedly, but $S_{A,B}(b_1,b_2)=\infty$ and $T_C(b_1,b_2)=\infty,$ for all log -bumps $A,B,C$ with $\bar A,\bar B,\bar C\in B_2.$
\end{theorem}
\begin{proof}
The idea is to construct a pair of  functions $(b_1, b_2)$ such that it satisfies the assumption in Theorem \ref{upperbound} so that we can conclude the boundedness of $C_bH$ directly, meanwhile, the related bump function increases slower than log-bumps.
Let
$\Phi_0=t^2\log(e+t)\log\log(e^e+t)^{3/2}$, and define
\[
b_1(x)=\sgn(x)1_{I_0}(x),\qquad b_2(x) = \sgn(x) \Phi_0^{-1}\big((M1_{I_0}(x))^{-1}\big),\ \ I_0=[-1,1].
\]We will show that $(b_1, b_2)$ is what we need. First of all, it is easy to check that for any cube $I$,
\begin{align*}
\langle |b_1|\rangle_{\Phi_0, I} \langle |b_2|\rangle_{\Phi_0,I}\le \Big\langle \Phi_0^{-1}\big((M1_{I_0}(x))^{-1}\big)^{-1} |b_2|\Big\rangle_{\Phi_0,I}\le 1.
\end{align*}
Then by the triangle inequality and general H\"older's inequality we have
\[
\langle |b_1-\langle b_1\rangle_I|\rangle_{\Phi_0, I} \langle |b_2-\langle b_2\rangle_I|\rangle_{\Phi_0,I}\lesssim 1.
\]
On the other hand, since $|b_1b_2|\le  1$, using triangle inequality and general H\"older's inequality again we have
\[
\langle |b_1-\langle b_1\rangle_I|  |b_2-\langle b_2\rangle_I|\rangle_{\Phi_0,I}\lesssim 1.
\]
Using Theorem \ref{upperbound} we know that $C_bH$ is bounded on $L^2$, thanks to $\bar\Phi_0\in B_2$.

It remains to show that $S_{A,B}(b_1,b_2)=\infty$ and $T_C(b_1,b_2)=\infty,$ for all log -bumps $A,B,C$ with $\bar A,\bar B,\bar C\in B_2.$
Without loss of generality we can assume that $A(t)=t^2 \log(e+t)^{1+\alpha}$, $B(t)=t^2 \log(e+t)^{1+\beta}$ and $C(t)=t^2 \log(e+t)^{1+\gamma}$, where $\alpha, \beta, \gamma>0$.
For $S_{A,B}(b_1, b_2)$ again we test with the interval $J_k=(-k,k)$ with $k\ge 2$. Since $b_1$ and $b_2$ are odd functions, we have
\begin{align*}
\langle |b_1-\langle b_1\rangle_{J_k}|\rangle_{A, J_k} \langle |b_2-\langle b_2\rangle_{J_k}|\rangle_{B,J_k}&=\langle |b_1|\rangle_{A, J_k} \langle |b_2|\rangle_{B,J_k}\\
&\simeq k^{-\frac 12}\log(e+k)^{\frac{1+\alpha}2} k^{\frac 12}\log(e+k)^{-\frac 12}\log\log(e^e+k)^{-\frac 34}\\
&\overset{k\rightarrow\infty}{\rightarrow}\infty.
\end{align*}
For $T_C(b_1, b_2)$, we test with the cube $I_k=(0,k)$, $k\ge 100$,  we have
\begin{align*}
	T_{C}(b_1, b_2)&\ge \lim_{k\rightarrow \infty}  \bla |b_1- \langle b_1\rangle_{I_k}| |b_2- \langle b_2\rangle_{I_k}| \bra_{ C, I_k} \ge \lim_{k\rightarrow \infty}  \bla |b_1- \langle b_1\rangle_{I_k}| |b_2- \langle b_2\rangle_{I_k}| 1_{(0,1)}\bra_{ C, I_k}\\
	&=\lim_{k\rightarrow \infty} (1-k^{-1})  \bla   |b_2- \langle b_2\rangle_{I_k}| 1_{(0,1)}\bra_{ C, I_k}\\
&\simeq \lim_{k\rightarrow \infty} (1-k^{-1})k^{\frac 12}\log(e+k)^{-\frac 12}\log\log(e^e+k)^{-\frac 34} k^{-\frac 12}\log(e+k)^{\frac{1+\alpha}2}=\infty.
	\end{align*}
\end{proof}

\begin{corollary}\label{notsuff} The conditions $S_{2+\varepsilon},T_{2+\varepsilon}$ are not precise enough to yield a characterization of $C_bH:L^2\to L^2.$
\end{corollary}
\begin{proof}
	The iterated commutator $C_bH$ of Theorem \ref{lastexample} is bounded on $L^2.$ We show that the conditions $S_{2+\varepsilon}(b_1,b_2)$ and $T_{2+\varepsilon}(b_1,b_2)$ are not satisfied for any $\varepsilon>0$.
 	To see this, it is enough to notice that for all log -bumps $A,B,C$ with $\bar A,\bar B,\bar C \in B_2,$ we have $t^{2+\varepsilon} \succeq A(t),B(t),C(t)$ by which by Proposition \ref{prop:young} iv) and the estimates in  Theorem \ref{lastexample} it follows that
	\begin{equation*}
		\langle\abs{b_1-\langle b_1\rangle_{J_k}}^{2+\varepsilon}\rangle_{J_k}^{\frac{1}{2+\varepsilon}}	\langle\abs{b_2-\langle b_2\rangle_{J_k}}^{2+\varepsilon}\rangle_{J_k}^{\frac{1}{2+\varepsilon}} \gtrsim \langle\abs{b_1-\langle b_1\rangle_{J_k}}\rangle_{A,J_k} \langle\abs{b_2-\langle b_2\rangle_{J_k}}\rangle_{B,J_k} \to \infty,
	\end{equation*}
	and
	\begin{equation*}
		\langle\abs{b_1-\langle b_1\rangle_{I_k}}^{2+\varepsilon}\abs{b_2-\langle b_2\rangle_{I_k}}^{2+\varepsilon}\rangle_{I_k}^{\frac{1}{2+\varepsilon}} \gtrsim 	\langle\abs{b_1-\langle b_1\rangle_{I_k}}\abs{b_2-\langle b_2\rangle_{I_k}}\rangle_{C,I_k} \to \infty,
	\end{equation*}
	as $k\rightarrow\infty$, showing that $S_{2+\varepsilon}(b_1,b_2) = \infty$ and $T_{2+\varepsilon}(b_1,b_2)=\infty.$	
\end{proof}

\begin{corollary}\label{altnotsuff} The commutator of Theorem \ref{lastexample} is bounded on $L^2$ and unbounded on all $L^p,$ $p\in(1,\infty)\setminus\{2\}.$
	\begin{proof}
	Let $p>2,$ $q\in(2,p)$ and
	$f(x) = x^{-1/q}1_{[100,\infty)}(x)\in L^p.$
	For all $x\in[-1,1],$
	\begin{equation*}
	\begin{split}
	&|C_bHf(x)| = \left|\int(b_1(x)-b_1(y))(b_2(x)-b_2(y))\frac{f(y)}{x-y}\ud y\right| \\
	&\sim \int_{100}^\infty \frac{y^{\frac 12}}{\log(e+y)^{\frac 12}\log\log(e^e+y)^{\frac 34}}\frac{y^{-\frac{1}{q}}}{y}\ud y = \infty,
	\end{split}
	\end{equation*}
	showing that $C_bH:L^p\not\to L^p.$ It follows  by duality that also $C_bH: L^{p'}\not\to L^{p'}.$
\end{proof}
\end{corollary}
\begin{remark}  Alternatively, we can prove Corollary \ref{notsuff} by Corollary \ref{altnotsuff}. Indeed, if the conditions $S_{2+\varepsilon}(b_1,b_2),T_{2+\varepsilon}(b_1,b_2)$ hold for some $\varepsilon >0,$ then by Remark \ref{altsuf} we have $C_bH:L^q\to L^q$ boundedly for all $q\in(2,2+\varepsilon),$ a contradiction with Corollary \ref{altnotsuff}.

\end{remark}
The above considerations lead us to conjecture:
\begin{conjecture}\label{conj}
With the functions $b_1,b_2$ subject to the same assumptions as those in Theorems \ref{upperbound} and \ref{lowerbound},
the boundedness of $[b_1, [b_2, H]]$ on $L^2(\R)$ is equivalent with the existence of Young functions $A,B, C$ with $\bar A, \bar B, \bar C\in B_2$ such that $S_{A,B}(b_1,b_2)+T_C(b_1, b_2)<\infty$.
\end{conjecture}

\end{document}